\documentclass[11pt]{amsart}

\usepackage{accents}
\usepackage{appendix}
\usepackage{amsfonts}
\usepackage{amsmath}
\usepackage{amssymb}	
\usepackage{amsthm}
\usepackage{array,booktabs,multirow}
\usepackage{braket}
\usepackage{cite}
\usepackage{dsfont}
\usepackage[shortlabels]{enumitem}
\usepackage{etoolbox}
\usepackage{float}
\usepackage[hang, flushmargin]{footmisc}
\usepackage{latexsym}
\usepackage{lipsum}
\usepackage{needspace}
\usepackage{tikz}
\usepackage{hyperref}
\usetikzlibrary{matrix,arrows}

\theoremstyle{plain}
\newtheorem{thm}{Theorem}[section]

\newtheorem{lem}[thm]{Lemma}
\newtheorem{prop}[thm]{Proposition}

\theoremstyle{definition}

\newtheorem{defn}[thm]{Definition}

\theoremstyle{remark}

\AtBeginEnvironment{lem}{\Needspace{3\baselineskip}}%
\AtBeginEnvironment{thm}{\Needspace{3\baselineskip}}
\AtBeginEnvironment{prop}{\Needspace{3\baselineskip}}
\AtBeginEnvironment{cor}{\Needspace{3\baselineskip}}

\setlist[enumerate,1]{leftmargin=2em}

\begin{document}

\title[Ihara zeta function, coefficients of Maclaurin series, and Ramanujan graphs]{Ihara zeta function, coefficients of Maclaurin series, and Ramanujan graphs}

\author{Hau-Wen Huang}
\address{
Department of Mathematics\\
National Central University\\
Chung-Li 32001 Taiwan
}
\email{hauwenh@math.ncu.edu.tw}

\begin{abstract}
Let $X$ denote a connected $(q+1)$-regular undirected graph of finite order $n$.
The graph $X$ is called Ramanujan whenever 
$$
|\lambda|\leq 2q^{\frac{1}{2}}
$$
for all nontrivial eigenvalues $\lambda$ of $X$. We consider the  variant  $\Xi(u)$ of the Ihara zeta function  $Z(u)$ of $X$ defined by 
\begin{gather*}
\Xi(u)^{-1}
=
 \left\{
 \begin{array}{ll}
 (1-u)(1-qu)(1-q^{\frac{1}{2}} u)^{2n-2}(1-u^2)^{\frac{n(q-1)}{2}}
 Z(u)
 \qquad
 &\hbox{if $X$ is nonbipartite},
 \\
 (1-q^2u^2)
 (1-q^{\frac{1}{2}} u)^{2n-4}
 (1-u^2)^{\frac{n(q-1)}{2}+1}
 Z(u)
 \qquad
 &\hbox{if $X$ is bipartite}.
 \end{array}
 \right.
\end{gather*}
The function $\Xi(u)$ satisfies the functional equation $\Xi(q^{-1} u^{-1})=\Xi(u)$. 
Let $\{h_k\}_{k=1}^\infty$ denote the number sequence given by
$$
\frac{d}{du}\ln \Xi(q^{-\frac{1}{2}}u)
=\sum_{k=0}^\infty h_{k+1} u^k.
$$
In this paper we establish the equivalence of the following statements: (i) $X$ is Ramanujan;
(ii) $h_k\geq 0$ for all $k\geq 1$; (iii) $h_{k}\geq 0$ for infinitely many even $k\geq 2$.
Furthermore we derive the Hasse--Weil bound for the Ramanujan graphs.
\end{abstract}
\maketitle



\section{Introduction}

The motivation of this paper originates from developing a graph theoretical counterpart of the following sufficient and necessary condition for the Riemann hypothesis. Recall that the Riemann zeta function $\zeta(s)$ is the analytic continuation of 
$$
\sum_{n=1}^\infty \frac{1}{n^s}.
$$
The negative even integers are trivial zeros of $\zeta(s)$ and the Riemann hypothesis asserts that the real part of every nontrivial zero of $\zeta(s)$ is $\frac{1}{2}$.
The Riemann xi function $\xi(s)$ is a variation of $\zeta(s)$ defined by 
$$
\xi(s)=\frac{1}{2}s(s-1)\pi^{-\frac{s}{2}}
\Gamma\left(\frac{s}{2}\right)
\zeta(s)
$$
where $\Gamma(s)$ is the Gamma function. The function $\xi(s)$ satisfies the functional equation 
$$ 
\xi(1-s)=\xi(s).
$$
Let $\{\lambda_k\}_{k=1}^\infty$ denote the number sequence given by 
$$
\frac{d}{dz}\ln \xi
\left(
\frac{1}{1-z}
\right)
=
\sum_{k=0}^\infty \lambda_{k+1} z^k.
$$
Li's criterion states that the Riemann hypothesis holds if and only if $\lambda_k\geq 0$ for all $k\geq 1$ \cite{Li:1997}.

Let $X$ denote an undirected graph of finite order allowing loops and multiple edges. 
We endow two opposite orientations on all edges of $X$ called the {\it oriented edges} of $X$. Given a vertex $x$ of $X$ the {\it valency} of $x$ is the number of oriented edges with the initial vertex $x$. If the valency of $x$ is equal to a constant $k$ for all vertices $x$ of $X$ then $X$ is said to be {\it $k$-regular}. 
The {\it adjacency matrix} $A$ of $X$ is a square matrix indexed by the vertices of $X$ whose $(x,y)$-entry is defined as the number of oriented edges from $x$ to $y$ 
for all vertices $x,y$ of $X$. 
Since $A$ is symmetric $A$ is diagonalizable with real eigenvalues. 
The eigenvalues of $A$ are also called the {\it eigenvalues} of $X$. 
A {\it walk} is a nonempty finite sequence of oriented edges which joins vertices. 
The {\it length} of a walk is the number of oriented edges in the walk.  
The graph $X$ is said to be {\it connected} if there exists a walk from $x$ to $y$ for any two distinct vertices $x,y$ of $X$.
A {\it cycle} is meant to be a walk from a vertex to itself. If all cycles on $X$ have even length then $X$ is called {\it bipartite}. 
A walk is said to have {\it backtracking} if an oriented edge is immediately followed by its opposite orientation in the walk. 
A cycle is said to be {\it geodesic} if all shifted cycles are backtrackless. 
For all $k\geq 1$ let $N_k$ denote the number of geodesic cycles on $X$ of length $k$. 
The {\it Ihara zeta function} $Z(u)$ of $X$ is the analytic continuation of 
\begin{gather}\label{zeta}
\exp
\left(\sum_{k=1}^\infty \frac{N_k}{k} u^k
\right).
\end{gather}

For the rest of this paper, we always assume that $X$ is a connected $(q+1)$-regular undirected graph of finite order $n$ with $q\geq 1$ and $n\geq 3$. 
In this case, the eigenvalues of $X$ with absolute value $q+1$ are called {\it trivial} eigenvalues and the poles of $Z(u)$ with values $\pm 1$ and $\pm q^{-1}$ are called {\it trivial} poles. The graph $X$ is said to be {\it Ramanujan} whenever 
$$
|\lambda|\leq 2q^{\frac{1}{2}}
$$
for all nontrivial eigenvalues $\lambda$ of $X$ \cite{Ramanujan1988}. The graph $X$ is Ramanujan if and only if all nontrivial poles of $Z(u)$ have the same absolute value $q^{-\frac{1}{2}}$, which is similar to the Riemann hypothesis by writing $u=q^{-s}$ \cite{Winnie2019,Terras2010}. We define the function $\Xi(u)$ by 
\begin{gather}\label{Xi}
\Xi(u)^{-1}
=
 \left\{
 \begin{array}{ll}
 (1-u)(1-qu)
 (1-q^{\frac{1}{2}}u)^{2n-2}
 (1-u^2)^{\frac{n(q-1)}{2}}
 Z(u)
 \qquad
 &\hbox{if $X$ is nonbipartite},
 \\
 (1-q^2u^2)
 (1-q^{\frac{1}{2}}u)^{2n-4}
 (1-u^2)^{\frac{n(q-1)}{2}+1}
 Z(u)
 \qquad
 &\hbox{if $X$ is bipartite}.
 \end{array}
 \right.
\end{gather}
The function $\Xi(u)$ is considered as the {\it Ihara xi function} which satisfies the functional equation 
$$
\Xi(q^{-1} u^{-1})=\Xi(u).
$$ 
If we set $u=q^{-s}$, then this becomes a functional equation relating $1-s$ and $s$ similar to the Riemann xi function.

\begin{defn}\label{defn:hk}
Let $\{h_k\}_{k=1}^\infty$ denote the number sequence given by
\begin{gather*}
\frac{d}{du}\ln \Xi(q^{-\frac{1}{2}} u)
=\sum_{k=0}^\infty h_{k+1} u^k.
\end{gather*}
\end{defn}

\noindent The main results of this paper are as follows:

\begin{thm}\label{thm:even}
If there is a positive even integer $k$ with $h_k\geq 0$ then 
$$
|\lambda|\leq
\left\{ 
\begin{array}{ll}
(1+\sqrt[k]{4 n -7})q^{\frac{1}{2}}
\qquad 
&\hbox{if $X$ is nonbipartite},
\\
(1+\sqrt[k]{2 n-7})q^{\frac{1}{2}}
\qquad 
&\hbox{if $X$ is bipartite}
\end{array}
\right.
$$
for all nontrivial eigenvalues $\lambda$ of $X$. 
\end{thm}

\begin{thm}\label{thm:Li}
The following are equivalent:
\begin{enumerate}
\item $X$ is Ramanujan.

\item $h_k\geq 0$ for all $k\geq 1$.

\item $h_k\geq 0$ for infinitely many even $k\geq 2$. 
\end{enumerate}
\end{thm}

\begin{thm}\label{thm:Hasse-Weil}
\begin{enumerate}
\item If $X$ is nonbipartite, then $X$ is Ramanujan if and only if 
\begin{align*}
|N_k-q^k-1| 
&\leq
2(n-1) q^{\frac{k}{2}}
\qquad 
\hbox{for all odd $k$};
\\
|N_k-n(q-1)-q^k-1| 
&\leq 
2(n-1) q^{\frac{k}{2}}
\qquad 
\hbox{for all even $k$}.
\end{align*}

\item If $X$ is bipartite, then $X$ is Ramanujan if and only if 
\begin{align*}
|N_k-n(q-1)-2q^k-2| 
&\leq 
2(n-2) q^{\frac{k}{2}}
\qquad 
\hbox{for all even $k$}.
\end{align*}
\end{enumerate}
\end{thm}

\noindent Theorem \ref{thm:even} is an improvement of the implication from Theorem \ref{thm:Li}(iii) to Theorem \ref{thm:Li}(i). The equivalence of Theorem \ref{thm:Li}(i), (ii) is an analogue of Li's criterion. In \cite[\S 1.3]{Iharazeta2018} it was shown that if $X$ is nonbipartite Ramanujan then 
$
N_k=
q^k+O(n q^{\frac{k}{2}})$; if $X$ is bipartite Ramanujan then 
$
N_k=
2q^k +O(n q^{\frac{k}{2}})
$ for even $k$. Theorem \ref{thm:Hasse-Weil} strengthens the above necessary conditions for $X$ as Ramanujan. Furthermore, Theorem \ref{thm:Hasse-Weil} is an analogue of the Hasse--Weil bound \cite{Weil1949}.

The paper is organized as follows: 
In \S\ref{s:Ihara} we give some preliminaries on $Z(u)$ and $\Xi(u)$. 
In \S\ref{s:hk} we derive three formulae for $\{h_k\}_{k=1}^\infty$. In \S\ref{s:proof} we prove Theorems \ref{thm:even}--\ref{thm:Hasse-Weil}. In \S\ref{s:behavior} we discuss the behavior of $\{h_{2k}\}_{k=1}^\infty$ when $X$ is not Ramanujan.

\section{The Ihara zeta and xi functions}\label{s:Ihara}

Let ${\rm Spec}(X)$ denote the {\it spectrum} of $X$; that is the multiset of all eigenvalues of $X$ with geometric multiplicities. Since $X$ is a connected $(q+1)$-regular undirected graph, the value $q+1\in {\rm Spec}(X)$ with multiplicity one.
Ihara's theorem \cite{Ihara:1966} states that $Z(u)$ is a rational function of the form 
\begin{gather}\label{zeta_v2}
Z(u)^{-1}=
(1-u^2)^{\frac{n(q-1)}{2}}\prod_{\lambda\in {\rm Spec}(X)}
1-\lambda u+qu^2.
\end{gather}
Substituting (\ref{zeta_v2}) into (\ref{Xi}) yields that 
\begin{gather}\label{Xi_v2}
\Xi^{-1}(u)=
\left\{
\begin{array}{ll}
\displaystyle
\frac{(1-u)(1-qu)(1-q^{\frac{1}{2}} u)^{2n-2}}{\prod\limits_{\lambda\in {\rm Spec}(X)}1-\lambda u+q u^2}
\qquad 
&\hbox{if $X$ is nonbipartite},
\\
\displaystyle
\frac{(1-u^2)(1-q^2u^2)(1-q^{\frac{1}{2}} u)^{2n-4}}{\prod\limits_{\lambda\in {\rm Spec}(X)}1-\lambda u+q u^2}
\qquad 
&\hbox{if $X$ is bipartite}.
\end{array}
\right.
\end{gather}
Let ${\rm Spec}^*(X)$ denote the multiset of all nontrivial eigenvalues of $X$ with geometric multiplicities. 
Recall that ${\rm Spec}(X)$ is symmetric with respect to $0$ when $X$ is bipartite. Hence
\begin{gather}\label{S*(X)}
{\rm Spec}^*(X)=
\left\{
\begin{array}{ll}
{\rm Spec}(X)\setminus\{q+1\}
\qquad 
&\hbox{if $X$ is nonbipartite},
\\
{\rm Spec}(X)\setminus\{\pm (q+1)\}
\qquad 
&\hbox{if $X$ is bipartite}.
\end{array}
\right.
\end{gather}  
Combined with (\ref{Xi_v2}) we obtain that
\begin{gather}\label{Xi_v3}
\Xi(u)=
\prod_{\lambda\in {\rm Spec}^*(X)}
\frac{1-\lambda u+q u^2}{(1-q^{\frac{1}{2}} u)^2}.
\end{gather}

\begin{prop}
$\Xi(u)$ satisfies the functional equation 
$\Xi(q^{-1} u^{-1})=\Xi(u)$.
\end{prop}
\begin{proof}
It is routine to verify the proposition by using (\ref{Xi_v3}). 
\end{proof}

\section{Formulae for $h_k$}\label{s:hk}

Recall the sequence $\{h_k\}_{k=1}^\infty$ from Definition \ref{defn:hk}. 
In this section we give two combinatorial formulae for $\{h_k\}_{k=1}^\infty$ and a formula for $\{h_k\}_{k=1}^\infty$ in terms of the Chebyshev polynomials. 

\begin{prop}\label{prop:formula1}
\begin{enumerate}
\item If $X$ is nonbipartite then 
\begin{gather*}
h_k
=
\left\{
\begin{array}{ll}
2(n-1)
+
q^{\frac{k}{2}}
+q^{-\frac{k}{2}}
-q^{-\frac{k}{2}} N_k 
\qquad 
&\hbox{for all odd $k$},
\\
2(n-1)
+
q^{\frac{k}{2}}
+q^{-\frac{k}{2}}
-
q^{-\frac{k}{2}}
\left(
N_k-n(q-1)
\right)
\qquad 
&\hbox{for all even $k$}.
\end{array}
\right.
\end{gather*}

\item If $X$ is bipartite then 
\begin{gather*}
h_k
=
\left\{
\begin{array}{ll}
2(n-2)
\qquad 
&\hbox{for all odd $k$},
\\
2(n-2+q^{\frac{k}{2}}
+ q^{-\frac{k}{2}})
-
q^{-\frac{k}{2}}
\left(
N_k-n(q-1)
\right) 
\qquad 
&\hbox{for all even $k$}.
\end{array}
\right.
\end{gather*}
\end{enumerate}
\end{prop}
\begin{proof}
Taking logarithm on (\ref{zeta}) yields that 
\begin{gather}\label{zeta_v1}
\ln Z(u)
=
\sum_{k=1}^\infty \frac{N_k}{k} u^k.
\end{gather}
Evaluate $h_k$ by using (\ref{Xi}) and (\ref{zeta_v1}) directly.
\end{proof}

 Let $\{T_k(x)\}_{k=0}^\infty$ denote the polynomials defined by 
\begin{align*}
xT_k(x) = T_{k+1}(x) + T_{k-1}(x)
\qquad \hbox{for all $k\geq 1$}
\end{align*}
with $T_0(x)=2$ and $T_1(x)=x$ \cite[\S 2.3]{centerAW:2016}. Note that $\frac{1}{2}T_k(2x)$ is  the $k$th Chebyshev polynomial of the first kind for all $k\geq 0$ \cite{chebyshev}.

\begin{lem}
[\!\!\cite{centerAW:2016, chebyshev}]
\label{lem1:Tk}
$T_k(x+x^{-1})=x^k+x^{-k}$ for all $k\geq 0$.
\end{lem}

Given a multiset $S$ of numbers and a constant $c$, we let $c S$ denote the multiset consisting of $cs$ for all $s\in S$.

\begin{prop}\label{prop:hk}
For all $k\geq 1$ the following equation holds:
$$
h_k=
2 |{\rm Spec}^*(X)|
-\sum_{s\in q^{-1/2} {\rm Spec}^*(X)}
T_k(s).
$$
\end{prop}
\begin{proof}
Let $S=q^{-\frac{1}{2}} {\rm Spec}^*(X)$. 
Applying (\ref{Xi_v3}) yields that 
\begin{align*}
\ln \Xi(q^{-\frac{1}{2}}u)
&=\sum_{s\in S}\ln (1-su+u^2)-2|S|\ln(1-u)
\\
&=\sum_{s\in S}\ln (1-su+u^2)+2|S|\sum_{k=1}^\infty \frac{u^k}{k}.
\end{align*}
Let $s\in S$ be given. Write $s=\alpha+\alpha^{-1}$ for some nonzero complex number $\alpha$. Then 
\begin{align*}
\ln (1-su+u^2)
&=
\ln(1-\alpha u)+\ln(1-\alpha^{-1} u)=
-\sum_{k=1}^\infty \frac{\alpha^k+\alpha^{-k}}{k} u^k.
\end{align*}
It follows from Lemma \ref{lem1:Tk} that $T_k(s)=\alpha^k+\alpha^{-k}$. By the above comments we have 
$$
\ln \Xi(q^{-\frac{1}{2}}u)=\sum_{k=1}^\infty \frac{2|S|-\sum_{s\in S} T_k(s)}{k} u^k.
$$
Now the proposition follows by taking differential on both sides of the above equation.
\end{proof}

For convenience we define ${-1\choose -1}=1$ and ${k\choose -1}=0$ for all $k\geq 0$.

\begin{lem}
[\!\!\cite{centerAW:2016,chebyshev}]
\label{lem3:Tk}
$
T_k(x) = 
\displaystyle
\sum_{i=0}^{\lfloor \frac{k}{2} \rfloor} 
(-1)^i
\left(
{k-i \choose i}+{k-i-1\choose i-1}
\right) x^{k-2 i}
$
for all $k\geq 0$. 
\end{lem}

For all $k\geq 1$ let $C_k$ denote the number of the cycles on $X$ of length $k$ and define $C_0=n$.

\begin{lem}\label{lem8:Tk}
$\sum\limits_{s\in q^{-1/2} {\rm Spec}(X)} T_k(s)
=
q^{-\frac{k}{2}}
\displaystyle
\sum_{i=0}^{\lfloor \frac{k}{2} \rfloor} 
(-q)^i
\left(
{k-i \choose i}+{k-i-1\choose i-1}
\right) C_{k-2 i}
$
for all $k\geq 0$.
\end{lem}
\begin{proof}
Evaluate the left-hand side by using Lemma \ref{lem3:Tk} along with the fact $C_k=
\sum_{\lambda \in {\rm Spec}(X)}\lambda^k
$ for all $k\geq 0$.
\end{proof}

Combining Propositions \ref{prop:formula1} and \ref{prop:hk} yields that 
\begin{gather}\label{e:Nk&Tk}
N_k=
\left\{
\begin{array}{ll}
q^\frac{k}{2} \sum_{s\in q^{-1/2} {\rm Spec}(X)} T_k(s) \qquad &\hbox{for all odd $k$},
\\
n(q-1)
+
q^\frac{k}{2} \sum_{s\in q^{-1/2} {\rm Spec}(X)} T_k(s)
\qquad &\hbox{for all even $k$}.
\end{array}
\right.
\end{gather}
As far as we know, the formula (\ref{e:Nk&Tk}) was first given in \cite[Lemma 4]{Iharazeta2018} and it was given a proof of \cite[Lemma 4]{Iharazeta2018} without using (\ref{zeta_v2}).
By Lemma \ref{lem8:Tk} and (\ref{e:Nk&Tk}) we have the following lemma:

\begin{lem}\label{lem:Nk}
The following equation holds:
\begin{gather*}
N_k
=
\left\{
\begin{array}{ll}
\displaystyle
\sum_{i=0}^{\frac{k-1}{2}} 
(-q)^i
\left(
{k-i \choose i}+{k-i-1\choose i-1}
\right)
C_{k-2i}
\qquad 
&\hbox{for all odd $k$},
\\
\displaystyle
n(q-1)+
\sum_{i=0}^{\frac{k}{2}} 
(-q)^i
\left(
{k-i \choose i}+{k-i-1\choose i-1}
\right)
C_{k-2i}
\qquad 
&\hbox{for all even $k$}.
\end{array}
\right.
\end{gather*}
\end{lem}

\begin{prop}\label{prop:formula2}
\begin{enumerate}
\item If $X$ is nonbipartite then 
\begin{gather*}
h_k
=
2(n-1)
+q^{\frac{k}{2}}+q^{-\frac{k}{2}}
-
q^{-\frac{k}{2}}
\sum_{i=0}^{\lfloor \frac{k}{2} \rfloor} 
(-q)^i
\left(
{k-i \choose i}+{k-i-1\choose i-1}
\right)
C_{k-2i}
\qquad 
\hbox{for all $k\geq 1$}. 
\end{gather*}

\item If $X$ is bipartite then 
\begin{gather*}
h_k
=
\left\{
\begin{array}{ll}
2(n-2)
\qquad 
&\hbox{for all odd $k$},
\\
\displaystyle
2(n-2+q^{\frac{k}{2}}+q^{-\frac{k}{2}})
-
q^{-\frac{k}{2}}
\sum_{i=0}^{\frac{k}{2}} 
(-q)^i
\left(
{k-i \choose i}+{k-i-1\choose i-1}
\right)
C_{k-2i}
\qquad 
&\hbox{for all even $k$}.
\end{array}
\right.
\end{gather*}
\end{enumerate}
\end{prop}
\begin{proof}
Combine Proposition \ref{prop:formula1} and Lemma  \ref{lem:Nk}.
\end{proof}

\section{Proof of the main results}\label{s:proof}

In this section we show Theorems \ref{thm:even}--\ref{thm:Hasse-Weil}.

\begin{lem}\label{lem:hk>=0}
For all $k\geq 1$ the coefficient $h_k\geq 0$ if and only if 
$$
\frac{1}{2|{\rm Spec}^*(X)|}\sum_{s\in q^{-1/2}{\rm Spec}^*(X)} T_k(s)\leq 1.
$$
\end{lem}
\begin{proof}
Immediate from Proposition \ref{prop:hk}.
\end{proof}

\begin{lem}[\!\!\cite{chebyshev}]
\label{lem2:Tk}
$T_k(2\cos \theta)=2 \cos k\theta$ for all $k\geq 0$ and all real numbers $\theta$. 
\end{lem}

\begin{lem}\label{lem5:Tk}
If $s$ is a real number with $|s|\leq 2$ then $|T_k(s)|\leq 2$ for all $k\geq 0$.
\end{lem}
\begin{proof}
Immediate from Lemma \ref{lem2:Tk}.
\end{proof}

\begin{lem}\label{lem7:Tk}
If $X$ is Ramanujan then $|T_k(s)|\leq 2$ for all $k\geq 0$ and all $s\in q^{-1/2}{\rm Spec}^*(X)$.
\end{lem}
\begin{proof}
Since $X$ is Ramanujan if and only if $|s|\leq 2$ for all $s\in q^{-1/2}{\rm Spec}^*(X)$, the lemma is immediate from Lemma \ref{lem5:Tk}.
\end{proof}

\begin{lem}
[\!\!\cite{chebyshev}]
\label{lem4:Tk}
$
T_k(x)=
\displaystyle{
2 
\left(\frac{x}{2}\right)^k
\sum_{i=0}^{\lfloor \frac{k}{2}\rfloor}
{k\choose 2i}
\left(
1-\left(\frac{x}{2}\right)^{-2}
\right)^i
}
$
for all $k\geq 0$.
\end{lem}

\begin{lem}\label{lem6:Tk}
If $s$ is a real number with $|s|> 2$ then $T_k(s)>0$ for all even $k\geq 0$.
\end{lem}
\begin{proof}
Immediate from Lemma \ref{lem4:Tk}.
\end{proof}

\begin{prop}\label{prop1}
Let $S$ denote a nonempty finite multiset consisting of real numbers. 
If  there is a positive even integer $k$ with
\begin{gather}\label{prop1:assumption}
\frac{1}{2|S|}\sum_{s\in S} T_k(s)\leq 1
\end{gather}
then $
|s|\leq 
1+\sqrt[k]{4|S|-3}
$
for all $s\in S$.
\end{prop}
\begin{proof}
For convenience let 
\begin{gather}\label{epsilon}
\varepsilon=
\frac{\sqrt[k]{4|S|-3}-1}{\sqrt[k]{4|S|-3}+1}.
\end{gather}
Suppose on the contrary that there is a real number $t\in S$ with 
\begin{gather}\label{contrary:|t|}
|t|>1+\sqrt[k]{4|S|-3}=\frac{2}{1-\varepsilon}.
\end{gather}
Using (\ref{contrary:|t|}) yields that 
$
1-\left(\frac{t}{2}\right)^{-2}
>
1-(1-\varepsilon)^2
=
2\varepsilon-\varepsilon^2$. 
Since $0<\varepsilon<1$ we have 
$
\varepsilon^2<2\varepsilon-\varepsilon^2$. 
Combined with Lemma \ref{lem4:Tk} this implies 
\begin{align*}
\frac{1}{2}T_k(t)
&>
\left(\frac{t}{2}\right)^k
\sum_{i=0}^{\lfloor \frac{k}{2}\rfloor}
{k\choose 2i} 
\varepsilon^{2i}
=
\left(\frac{t}{2}\right)^k
\frac{(1+\varepsilon)^k+(1-\varepsilon)^k}{2}.
\end{align*}
By Lemmas \ref{lem5:Tk} and \ref{lem6:Tk} we have 
$
T_k(s)\geq -2$ for all $s\in S\setminus\{t\}$.

Combining (\ref{prop1:assumption}) with the above comments yields that 
$$
|S|\geq \frac{1}{2}  \sum_{s\in S} T_k(s)
=\frac{1}{2} T_k(t)+ \frac{1}{2} \sum_{s\in S\setminus\{t\}}  T_k(s)
>
\left(\frac{t}{2}\right)^k
\frac{(1+\varepsilon)^k+(1-\varepsilon)^k}{2}-(|S|-1).
$$
It leads to  
\begin{gather}\label{|t|_upperbound}
\left(\frac{t}{2}\right)^k< \frac{4|S|-2}{(1+\varepsilon)^k+(1-\varepsilon)^k}.
\end{gather}
Since $k$ is even the inequality (\ref{contrary:|t|}) implies that 
\begin{gather}\label{|t|_lowerbound}
\left(\frac{t}{2}\right)^k>\frac{1}{(1-\varepsilon)^k}.
\end{gather}
Combining (\ref{|t|_upperbound}) and (\ref{|t|_lowerbound}) we see that 
\begin{gather}\label{epsilon_ineq}
\frac{1}{(1-\varepsilon)^k}<\frac{4|S|-2}{(1+\varepsilon)^k+(1-\varepsilon)^k}.
\end{gather}
Using the setting (\ref{epsilon}) it is routine to verify that both sides of (\ref{epsilon_ineq}) are equal, 
a contradiction. The proposition follows.
\end{proof}

\noindent{\it Proof of Theorem \ref{thm:even}.} 
By Lemma \ref{lem:hk>=0}, when $X$ is nonbipartite the result follows by applying Proposition \ref{prop1} with $S$ replaced by $q^{-1/2}{\rm Spec}^*(X)$.

Note that the number of zeros in ${\rm Spec}^*(X)$ is even if the regular graph $X$ is bipartite. 
Since $k$ is even it follows from Lemma \ref{lem1:Tk} that $T_k(x)$ is an even function. Combined with Lemma \ref{lem:hk>=0}, when $X$ is bipartite the result follows by applying Proposition \ref{prop1} with $S$ chosen as the multiset of all positive numbers and a half number of zeros in $q^{-1/2}{\rm Spec}^*(X)$. 
 \hfill $\square$

\medskip

\noindent{\it Proof of Theorem \ref{thm:Li}.}
(i) $\Rightarrow$ (ii):  Combine Lemmas \ref{lem:hk>=0} and  \ref{lem7:Tk}.

(ii) $\Rightarrow$ (iii): It is obvious.

(iii) $\Rightarrow$ (i): Immediate from Theorem \ref{thm:even}. 
\hfill $\square$

\begin{lem}\label{lem:hk_upper}
If $X$ is Ramanujan then 
$$
h_k
\leq 
\left\{
\begin{array}{ll}
4(n-1)
\qquad 
\hbox{if $X$ is nonbipartite},
\\
4(n-2)
\qquad 
\hbox{if $X$ is bipartite}
\end{array}
\right.
$$
for all $k\geq 1$.
\end{lem}
\begin{proof}
By Proposition \ref{prop:hk} and Lemma \ref{lem7:Tk} the coefficient $h_k\leq 4|{\rm Spec}^*(X)|$ for all $k\geq 1$. Hence the lemma follows by (\ref{S*(X)}).
\end{proof}

\noindent{\it Proof of Theorem \ref{thm:Hasse-Weil}.}  Combine Theorem \ref{thm:Li}(i), (ii) and Proposition \ref{prop:formula1} along with Lemma \ref{lem:hk_upper}. 
\hfill $\square$

\section{Behavior of $h_{2k}$}\label{s:behavior}

We end this paper with a remark on $\{h_{2k}\}_{k=1}^\infty$ under the assumption that $X$ is not Ramanujan. 
Let $S$ denote the set of all nonzero complex numbers $\alpha$ with $\alpha+\alpha^{-1}\in q^{-1/2}{\rm Spec}^*(X)$. 
For those $s\in q^{-1/2}{\rm Spec}^*(X)$ with $|s|\leq 2$ the corresponding numbers $\alpha\in S$ have the same absolute value $1$. By the assumption that $X$ is not Ramanujan there exists an $s\in q^{-1/2}{\rm Spec}^*(X)$ with $|s|>2$ and the corresponding numbers $\alpha \in S$ are real and $\alpha\not=\pm 1$.
Let 
$$
\mu=\max_{\alpha\in S} |\alpha|>1. 
$$
Since the function $f(x)=x+x^{-1}$ is strictly increasing on $(1,\infty)$, it follows that 
\begin{align*}
\mu+\mu^{-1}
&=\max_{\alpha\in S}|\alpha|+|\alpha|^{-1}
=\max_{s\in q^{-1/2}{\rm Spec}^*(X)} |s|.
\end{align*}
Using Lemma \ref{lem1:Tk} yields that 
$$
\lim_{k\to \infty}\frac{T_{2k}(\alpha+\alpha^{-1})}{\mu^{2k}}
=\left\{
\begin{array}{ll}
0 
\qquad 
&\hbox{if $\mu>|\alpha|$ and $\mu>|\alpha|^{-1}$},
\\
1
\qquad 
&\hbox{else}
\end{array}
\right.
$$
for all $\alpha\in S$. It follows from Proposition \ref{prop:hk} that 
$h_{2k}$ is asymptotic to $-m \mu^{2k}$ as $k$ approaches to $\infty$, 
where $m$ is the number of $s\in q^{-1/2}{\rm Spec}^*(X)$ with $|s|=\mu+\mu^{-1}$. Therefore the following result holds:

\begin{thm}\label{thm:h2k}
If $X$ is not Ramanujan then 
\begin{gather*}\label{h2k}
q^{-\frac{1}{2}}\max |\lambda|
=\lim_{k\to \infty}\sqrt{\frac{h_{2k+2}}{h_{2k}}}+\sqrt{\frac{h_{2k}}{h_{2k+2}}}
\end{gather*}
where $\max$ is over all nontrivial eigenvalues $\lambda$ of $X$. 
\end{thm}

\subsection*{Acknowledgements}

The research is supported by the Ministry of Science and Technology of Taiwan under the project MOST 106-2628-M-008-001-MY4.

\bibliographystyle{amsplain}
\bibliography{MP}

\providecommand{\bysame}{\leavevmode\hbox to3em{\hrulefill}\thinspace}
\providecommand{\MR}{\relax\ifhmode\unskip\space\fi MR }
\providecommand{\MRhref}[2]{%
  \href{http://www.ams.org/mathscinet-getitem?mr=#1}{#2}
}
\providecommand{\href}[2]{#2}
\begin{thebibliography}{1}

\bibitem{centerAW:2016}
H.-W. Huang, \emph{{Center of the universal Askey--Wilson algebra at roots of
  unity}}, Nuclear Physics B \textbf{909} (2016), 260--296.

\bibitem{Ihara:1966}
Y.~Ihara, \emph{{On discrete subgroups of the two by two projective linear
  group over $\mathfrak p$-adic fields}}, Journal of the Mathematical Society
  of Japan \textbf{18} (1966), 219--235.

\bibitem{Winnie2019}
W.-C.~W. Li, \emph{{Zeta and $L$-functions in number theory and
  combinatorics}}, Regional Conference Series in Mathematics 129, American
  Mathematical Society, Providence, RI, 2019.

\bibitem{Li:1997}
X.-J. Li, \emph{{The positivity of a sequence of numbers and the Riemann
  hypothesis}}, Journal of Number Theory \textbf{65} (1997), 325--333.

\bibitem{Ramanujan1988}
A.~Lubotzky, R.~Phillips, and P.~Sarnak, \emph{{Ramanujan graphs}},
  Combinatorica \textbf{8} (1988), 261--277.

\bibitem{chebyshev}
J.~C. Mason and D.~C. Handscomb, \emph{{Chebyshev polynomials}}, Chapman and
  Hall/CRC, New York, 2003.

\bibitem{Iharazeta2018}
B.~Rangarajan, \emph{{A combinatorial proof of Ihara--Bass's formula for the
  zeta function of regular graphs}}, {In 37th IARCS Annual Conference on
  Foundations of Software Technology and Theoretical Computer Science (FSTTCS
  2017).} (Germany), Schloss Dagstuhl--Leibniz-Zentrum f\"{u}r Informatik,
  Dagstuhl Publishing, 2018.

\bibitem{Terras2010}
A.~A. Terras, \emph{{Zeta functions of graphs: A stroll through the garden}},
  Cambridge Studies in Advanced Mathematics 128, Cambridge University Press,
  Cambridge, UK; New York, 2010.

\bibitem{Weil1949}
A.~Weil, \emph{{Numbers of solutions of equations in finite fields}}, Bulletin
  of the American Mathematical Society \textbf{55} (1949), 497--508.

\end{thebibliography}

\end{document}